\documentclass[reqno,12pt]{article}

\usepackage{a4wide}
\usepackage{amsmath,stmaryrd} 
\usepackage{amssymb}
\usepackage{amsthm}
\usepackage[utf8]{inputenc} 
\usepackage{graphicx} 
\usepackage{xcolor}
\usepackage[normalem]{ulem}

\numberwithin{equation}{section}

\newtheorem{theo}{Theorem}

\newtheorem{coro}{Corollary}
\newtheorem{prop}{Proposition}

\theoremstyle{remark}
\newtheorem{Remark}{Remark}

\newtheorem*{Remark*}{Remark}
\newtheorem*{Remarks*}{Remarks}

\makeatletter
\newcommand*{\house}[1]{%
 \mathord{%
 \mathpalette\@house{#1}%
 }%
}
\newcommand*{\@house}[2]{%
 \dimen@=\fontdimen8 %
 \ifx#1\scriptscriptstyle\scriptscriptfont
 \else\ifx#1\scriptstyle\scriptfont
 \else\textfont\fi\fi
 3 %
 \sbox0{%
 $#1%
 \vrule width\dimen@\relax
 \overline{%
 \kern2\dimen@
 \begingroup 
 #2%
 \endgroup
 \kern2\dimen@
 }%
 \vrule width\dimen@\relax
 \mathsurround=1.5\dimen@ 
 $%
 }%
 \ht0=\dimexpr\ht0-\dimen@\relax
 \dp0=\dimexpr\dp0+2\dimen@\relax
 \vbox{%
 \kern\dimen@ 
 \copy0 %
 }%
}

\newcommand{\housesurQ}[1]{|#1|}

\newcommand{\tra}{{}^t}
\newcommand{\N}{\mathbb{N}}
\newcommand{\Z}{\mathbb{Z}}
\newcommand{\Q}{\mathbb{Q}}
\newcommand{\R}{\mathbb{R}}

\newcommand{\Qbar}{\overline{\mathbb Q}}

\newcommand{\etoile}{^*}

\newcommand{\ord}{{\rm ord}}

\begin{document}

\title{Transcendence of values of logarithms of $E$-functions}
\date\today
\author{S. Fischler and T. Rivoal}
\maketitle

\begin{abstract} Let $f$ be an $E$-function (in Siegel's sense) not of the form $e^{\beta z}$, $\beta \in \Qbar$, and let $\log$ denote any fixed determination of the complex logarithm. We first prove that there exists a finite set $S(f)$ such that for all $\xi\in \Qbar\setminus S(f)$, $\log(f(\xi))$ is a transcendental number. We then quantify this result when $f$ is an $E$-function in the strict sense with rational coefficients, by proving an irrationality measure of $\ln(f(\xi))$ when $\xi\in \mathbb Q\setminus S(f)$ and $f(\xi)>0$. This measure implies that $\ln(f(\xi))$ is not an ultra-Liouville number, as defined by Marques and Moreira. The proof of our first result, which is in fact more general, uses in particular a recent theorem of Delaygue. The proof of the second result,  which is independent of the first one, is a consequence of a new linear independence measure for values of linearly independent $E$-functions in the strict sense with rational coefficients, where emphasis is put on other parameters than on the height, contrary to the case in Shidlovskii's classical measure for instance.
\end{abstract}

\section{Introduction}

In this paper, we pursue our study of the properties of the values of $E$-functions at algebraic points, and more specifically of the logarithms of these values. We recall the definition of $E$-functions. As usual, we embed $\Qbar$ in $\mathbb C$. 
A power series $F(z)=\sum_{n=0}^{\infty} a_n z^n/n! \in \Qbar[[z]]$ is said to be a {\em strict} $E$-function~if 

$(i)$ $F(z)$ is solution of a non-zero linear differential equation with coefficients in $\Qbar(z)$.
\smallskip

$(ii)$ There exists $C_1>0$ such that all Galois conjugates of $a_n$ have modulus $\le C_1^{n+1}$, for all $n\ge 0$.

\smallskip

$(iii)$ There exists $C_2>0$ and a sequence of positive integers $d_n$, with $d_n \leq C_2^{n+1}$, such that $d_na_m$ are algebraic integers for all~$m\le n$.

\noindent In fact, $E$-functions have been defined by Siegel~\cite{siegel} in 1929 in a more general way, ie the two bounds $(\cdots)\le C_i^{n+1}$ are replaced by: for all $\varepsilon>0$, $(\cdots)\le n!^{\varepsilon}$ for all $n\ge N(\varepsilon)$. It is believed that an $E$-function in Siegel's sense is automatically a strict $E$-function; see \cite[p.~715]{andre} for a discussion. 
Unless otherwise specified, $E$-functions below will be understood in Siegel's sense. Note that if $a_n\in \mathbb Q$, $(ii)$ and $(iii)$ read $\vert a_n\vert \le C^{n+1}$ and $d_na_m\in \mathbb Z$; in $(i)$, there exists such a differential equation with coefficients in $\mathbb Q(z)$, and the normalized one of minimal order also has coefficients in $\mathbb Q(z)$. 
An $E$-function is either a polynomial or a transcendental function.

\medskip

 Let $f$ and $g$ be two $E$-functions. If $f$ is transcendental and $g$ is a polynomial, then $f(\Qbar)\cap g(\Qbar)$ is finite by \cite{ar}. If $f$ and $g$ are polynomials, two cases occur: if one of them is a constant, $f(\Qbar)\cap g(\Qbar)$ is finite, while if none is a constant, $f(\Qbar)\cap g(\Qbar)$ is infinite. Our first result completes the picture; it shows that a transcendental $E$-function is determined by the set of values it takes at algebraic numbers.
\begin{theo} \label{theo:1} 
Let $f,g$ be two transcendental $E$-functions such that $f(z)$ is not of the form $g(\beta z)$, $\beta\in \Qbar$. Then 
 $\{(\xi,\eta)\in \Qbar^2 : f(\xi)=g(\eta)\}$ is a finite set.
Equivalently, 
$f(\Qbar)\cap g(\Qbar)$ is a finite set.
\end{theo}
The equivalence is a consequence of Proposition~\ref{prop:1} proved in~\S\ref{sec:1}. 
The assumption in Theorem~\ref{theo:1} is obviously also necessary to obtain finiteness when $f$ and $g$ are transcendental. 

Applying Theorem~\ref{theo:1} to $f(z)$ transcendental and $g(z)=e^{z}$, we deduce that the set $\{(\xi, \eta)\in \Qbar^2 : f(\xi)=e^{\eta}\}$ is finite when $f(z)\neq e^{\beta z}$ for all $\beta \in \Qbar$. As $\{(\xi, \eta)\in \Qbar^2 : f(\xi)=e^{\eta}\}$ is also finite if $f\in \Qbar[z]$ by the Hermite-Lindemann Theorem, we obtain the following result. 
\begin{coro} \label{coro:2} Let $f$ be an $E$-function not of the form $e^{\beta z}$, $\beta \in \Qbar$, and let $\log$ denote any fixed determination of the complex logarithm. There exists a finite set $S(f)$ such that for all $\xi\in \Qbar\setminus S(f)$, $\log(f(\xi))$ is a transcendental number. 
\end{coro}
As the proof shows, given $\xi\in\Qbar$, there exists an algebraic determination of the logarithm of $f(\xi)$ if, and only if, either $f(\xi)=1$ or the $E$-function $f(z)-\exp( z/ \varrho )$ vanishes at $\xi$ for some $\varrho$ in the finite set $ \mathfrak S(f) $ considered in \S \ref{subsecLW}, which determines $S(f)$.

This corollary applies to any $E$-function with a minimal differential equation of order~$\ge 2$, for example to Bessel's function $J_0(z):=\sum_{n=0}^\infty (-1)^n (z/2)^{2n}/n!^2$ whose minimal equation is  $zy''(z)+y'(z)+zy(z)=0$. But the property $\{(\xi,\eta)\in \Qbar^2 : J_0(\xi)=e^{\eta}\}=\{(0,0)\}$ is not a new result as it is a consequence of the much more general \cite[p.~219, Theorem~4]{shid}: for any $\xi, \eta\in \Qbar^*$, the numbers $J_0(\xi), J_0'(\xi)$ and $e^{\eta}$ are algebraically independent over $\Qbar$. 
Other examples of a similar flavor involving generalized hypergeometric series ${}_pF_q[z^{q-p+1}]$ with rational parameters satisfying certain arithmetic conditions can be deduced from the very general algebraic independence result in \cite[p.~300, Corollary~4.6]{bbh}. However, these conditions do not exhaust all such series with rational parameters. To the best of our knowledge, Corollary~\ref{coro:2} is new for $f(z):={}_pF_q[z^{q-p+1}]$ with $1\le p \le q$, $1/2$ as a lower parameter and no upper parameter equal to $1/2 \mod \mathbb Z$, because neither assumption A) nor assumption B) on page 280 of \cite{bbh} is satisfied, for instance ${}_1F_1[1/3;1/2;z]$.

\medskip

Corollary~\ref{coro:2} can be quantified in the rational and strict case.  Our method to prove Theorem \ref{theo:2} is independent though of that of Theorem \ref{theo:1} (based on a recent result of Delaygue \cite{Delaygue}), as it uses a new linear independence measure for values of $E$-functions (Proposition~\ref{prop:1} stated in \S\ref{subsecenoncemesure}); see below for more details.  When $x$ is a positive real number, we denote by $\ln(x)$ its napierian logarithm. 
\begin{theo} \label{theo:2} 
Let $f\in \mathbb Q[[z]]$ be a strict $E$-function, and 
 $\xi\in \mathbb Q^*$ be such that $f(\xi)>0$ and $ \ln(f(\xi))\not\in\Q$. Then there exist $c,d>0$ such that, for all $(a,b)\in \mathbb Z\times \mathbb N^*$, 
\begin{equation}\label{eq:irratmeasurelogE}
\left\vert \ln(f(\xi))-\frac{a}{b}\right\vert \ge \frac{1}{\exp({c b^d})}.
\end{equation}
\end{theo}

In particular, this result implies that $\ln(f(\xi))$ is not an ultra-Liouville number (as defined in \cite{MM}).

Theorem~\ref{theo:2} applies to any $\xi\in \mathbb Q^*$ such that $J_0(\xi)>0$ because 
 $\ln(J_0(\xi))$ is then a transcendental number by the above mentioned result; the irrationality measure for $\ln(J_0(\xi))$ is new to our knowledge. Note that Theorem~\ref{theo:2} can also be applied to any $f\in \mathbb Q[z]\setminus\{0\}$, but in this case the lower bound can be much improved because it is known that for any $\alpha\in \mathbb Q_{>0}$, $\alpha\neq 1$, the number $\ln(\alpha)$ is an irrational number and not a Liouville number; see \cite[p.~150, Satz~5]{mahler}. If $f(\xi)<0$, the same result holds for $-f(\xi)$ instead of $f(\xi)$.

The constants $c,d$ in \eqref{eq:irratmeasurelogE} depend on $f$ and $\xi$; they are effective but we did not try to compute them (this could be done in principle) because it is likely that the lower bound in \eqref{eq:irratmeasurelogE} is not optimal and could be replaced by $c/b^d$ for some other effective constants $c,d>0$, proving that $\ln(f(\xi))$ is not a Liouville number.~It does not seem that this improvement could be obtained with our method, which relies on the following observation: 
\begin{center}
 $\ln(f(\xi))$ is close to $a/b$ if, and only if, $f(\xi)-\exp(a/b)$ is small.
\end{center}
Considering the latter as a linear combination (with coefficients 1 and $-1$) of values of $E$-functions, one may try to apply linear independence measures due to Shidlovskii~\cite{shid}. A first problem is that such measures are optimized for linear forms with very large coefficients, which is not the case here. But another more important problem arises: the constants in such measures depend (usually in a non-explicit way) on the functions involved, and in our setting $\exp(az/b)$ is amongst them. For this reason we need a linear independence measure in which the dependencies of the constants in terms of the functions are explicit. We prove such a result, namely Proposition~\ref{propmino}, in \S\ref{sec:2}.

We note that using instead Brownawell's more general algebraic independence measure~(\footnote{Brownawell's measure is slightly ineffective because of the ineffectivity of Shidlovskii's constant $n_0$, that will also be used in the present paper. As we shall explain, $n_0$ can now be bounded effectively, removing any form of ineffectivity in Shidlovskii's and Brownawell's measures.}) in \cite{brownawell} (and making explicit in it the dependence on the parameters we need, as with Shidlovskii's measure), we only found a lower bound of the form $1/\!\exp(\exp(c b^d))$ on the right-hand side of \eqref{eq:irratmeasurelogE}, provided that $f, f', \ldots, f^{(\eta-1)}$ are algebraically independent (where $\eta$ is the minimal order of a non-trivial differential equation satisfied by a strict $E$-function $f\in \mathbb Q[[z]]$).

\medskip

The structure of this paper is the following. In Section \ref{sec:1} we prove Theorem~\ref{theo:1} using Delaygue's analogue of the Lindemann-Weierstrass theorem. Then we state and prove in Section \ref{sec:2} the linear independence measure that we shall use in Section \ref{sec:3} to prove Theorem~\ref{theo:2}.

\section{Proof of Theorem~\ref{theo:1}} \label{sec:1}

\subsection{Delaygue's analogue of the Lindemann-Weierstrass theorem}\label{subsecLW}

Given an $E$-function $f(z)=\sum_{n=0}^\infty \frac{a_n}{n!}z^n $, let $\mathfrak S(f)\subset\Qbar\etoile$ denote the set of finite singularities of the $G$-function $ \sum_{n=0}^\infty a_n z^{n} $. 

\medskip

We shall derive Theorem~\ref{theo:1} from the following special case of Delaygue's analogue of the linear version of the Lindemann-Weierstrass theorem (see \cite[Corollary~2.2]{Delaygue}).  

\begin{prop}\label{propLW} Let $f_1$, $f_2$ be $E$-functions and $z_1$, $z_2$ be non-zero algebraic numbers such that $f_1(z_1) = f_2(z_2)$ is transcendental. Then $z_1/z_2$ can be written as $\varrho_1/\varrho_2$ with $\varrho_1\in \mathfrak S(f_1) $ and $\varrho_2\in \mathfrak S(f_2) $.
\end{prop}
The important point for us, in the conclusion of Proposition~\ref{propLW}, is that $z_1/z_2$ belongs to a finite set determined by $f_1$ and $f_2$.
\begin{Remark} Proposition~\ref{propLW} is proved in \cite{Delaygue} for $E$-functions in the strict sense. However, it also holds for $E$-functions in the large sense (and so does the general result \cite[Theorem~2.1]{Delaygue}) because only the following properties are used in the proof, and they hold for $G$-functions and $E$-functions in the large sense by the results proved in \cite{andre2, lepetit}:
\begin{enumerate}
 \item[(1)] The point at infinity is regular or a regular singularity of any  $G$-function, because it is solution of a fuchsian differential operator.
 \item[(2)] A   $G$-function without finite singularity is a polynomial. Indeed such a function is entire, and has moderate growth at infinity by (1). By Liouville's theorem, it is a polynomial.
 \item[(3)] Any   $E$-function is annihilated by an $E$-operator, without non-zero finite singularity.
 \item[(4)] Beukers' refined version of the Siegel-Shidlovskii theorem (\textit{i.e.}~\cite[Theorem~1.3]{beukers}) holds.
\end{enumerate}

\end{Remark}

\subsection{Application to Theorem~\ref{theo:1}}

We first prove 
\begin{prop}\label{prop:1} Let $f$ be a non-constant $E$-function and $\chi\in\mathbb C$. Then the set $\{\alpha \in \Qbar : f(\alpha)=\chi\}$ is finite.
\end{prop}

\begin{proof}[Proof of Proposition~\ref{prop:1}]
If $\chi\in \Qbar$, this is a consequence of the main result in \cite{ar}. Otherwise, let us fix $\alpha_0\in\Qbar$ such that $f(\alpha_0)=\chi$; if there is no such $\alpha_0$, the corresponding set is empty and therefore finite. For any $\alpha \in\Qbar$ such that $f(\alpha )=\chi$, Proposition~\ref{propLW} implies that $\alpha/\alpha_0$ belongs to a finite set determined by $f$. This concludes the proof of Proposition~\ref{prop:1}.
\end{proof}

\begin{proof}[Proof of Theorem~\ref{theo:1}] 
First of all, let us consider the set of pairs $(\xi,\eta)\in\Qbar^2$ such that $f(\xi)=g(\eta)$ is algebraic. Recall from \cite{ar} that a transcendental $E$-function takes algebraic values at only finitely many algebraic points. Therefore each of $\xi$, $\eta$ belongs to a finite set determined by $f$ and $g$: so does the pair $(\xi,\eta)$.

Now let us move to pairs $(\xi,\eta)\in\Qbar^2$ such that $f(\xi)=g(\eta)$ is transcendental; this implies $\xi,\eta\neq 0$. Then Proposition~\ref{propLW} provides a finite set (determined by $f$ and $g$) that contains all quotients $\eta/\xi$. For each value $\beta$ of the quotient, the $E$-function $f(z)-g(\beta z)$ vanishes at $\xi$. Since this $E$-function is not identically zero by hypothesis, $\xi$ belongs to a finite set determined by $\beta$. So does $\eta$, and this concludes the proof that 
$\{(\xi,\eta)\in \Qbar^2 : f(\xi)=g(\eta)\}$ is a finite set. This is equivalent to the fact that $I:=f(\Qbar)\cap g(\Qbar)$ is a finite set. Indeed, if $\{(\xi,\eta)\in \Qbar^2 : f(\xi)=g(\eta)\}$ is finite, then obviously $I$ is finite. Conversely, if $I$ is finite, then for any $\chi\in I$, there are only finitely many $(\xi,\eta)\in \Qbar^2$ such that $f(\xi)=\chi=g(\eta)$ by Proposition~\ref{prop:1}. This completes the proof of Theorem~\ref{theo:1}. 
\end{proof}

\section{An explicit version of Shidlovskii's linear independence measure} \label{sec:2}

In this section, we prove a linear independence measure of values of $E$-functions, namely Proposition~\ref{propmino} stated in \S \ref{subsecenoncemesure}. The important point for our application to Theorem~\ref{theo:2} is that when the last function is $\exp(\beta z)$, the constants are controlled in terms of $\beta$: this special case is studied in Corollary~\ref{coromino}.

The structure of proof is similar to that of Shidlovskii's measure, so we recall it in \S \ref{ssec:21}. We proceed to the proof in \S \ref{ssec:22}.

\subsection{Statement of the measure}\label{subsecenoncemesure}

Assume $Y:={}^t(f_1, \ldots, f_m)$ is a vector of strict $E$-functions in $\mathbb Q[[z]]$, solution of a differential system $Y'=AY$ with $A\in M_m(\mathbb Q(z))$; assume moreover that $f_1, \ldots, f_m$ are $\mathbb Q(z)$-linearly independent. Let $ T\in \mathbb Z[z]\setminus\{0\}$ be a common denominator of minimal degree of the entries of $A:= (A_{i,j})_{1\le i,j\le m}$. Let $\xi\in \mathbb Q$ be such that $\xi T(\xi)\neq 0$. We consider any integer $H\ge 1$ and any vector $(a_1, \ldots, a_m)\in \mathbb Z^m\setminus \{0\}$ such that $\max\vert a_k\vert \le H$.

\medskip

As a special case $\mathbb K=\mathbb Q$ of \cite[p.~357,~Theorem~1]{shid}, Shidlovskii proved that for any $\varepsilon>0$, there exists an ineffective constant $c>0$ such that, in the above situation, we have 
\begin{equation}\label{eq:mesindepshid}
\bigg\vert \sum_{j=1}^m a_j f_j(\xi) \bigg\vert > \frac{c}{H^{m-1+\varepsilon}}.
\end{equation}
This constant $c$ is now effective (because the integer $n_0$ in Shidlovskii's multiplicity estimate can be bounded, see below). However it depends on $f_1$, \ldots, $f_m$ in a way which is not made explicit by Shidlovskii. This is a problem to prove Theorem~\ref{theo:2}, since in our setting 
$f_m$ will be $\exp(\beta z)$ and we need constants that we control explicitly in terms of $\beta$. 

We shall prove a linear independence measure, namely Proposition~\ref{propmino}, in which the dependencies of the constants on certain parameters important for us are made explicit, unlike the classical measures in the same context. With this aim in mind, we define as in \cite[p.~93]{shid}: 
\begin{equation}\label{eq:defpq}
p:=\min_{1\le j \le m}\ord_{z=0}f_j(z)
\quad \mbox{ and } \quad 
q:
=\max(\deg T, \max_{i,j} \deg(TA_{i,j})).
\end{equation}

Since the $f_i=\sum_{k=0}^\infty \frac{\varphi_{k,i}}{k!}z^k$, $i=1,\ldots, m$, are $E$-functions in the strict sense, there exists a constant $C>0$ such $\vert \varphi_{k,i}\vert \le C^{k+1}$ for all $k\ge 0, i\in\{1, \ldots, m\}$, and there exists a constant $D>0$ such the common denominator $d_{k,i}$ of $\varphi_{0,i}, \ldots, \varphi_{k,i}$ satisfies $d_{k,i}\le D^{k+1}$ for all $k\ge 0, i\in\{1, \ldots, m\}$. 

We also denote by $E$ the maximum modulus of the coefficients of the polynomial $T(z)$ and of all the polynomials $T(z)A_{i,j}(z)$.

We denote by $n_0\ge 0$ the constant in Shidlovskii's zero estimate. Shidlovskii's proof of the existence of $n_0$ is not effective (see the proof of \cite[p.~93, Lemma~8]{shid} and the definition of $n_0$ in \cite[p.~99, Eq.~(83)]{shid}). However, following the works of Bertrand, Beukers, Chirskii and Yebbou \cite{bb, bcy}, it is now known that the integer $n_0$ can be bounded above using explicit quantities that depend on the matrix $A$ of the differential system. Certain of these quantities are themselves bounded by means of the generalized local exponents at the singularities of $A$ and the point at infinity (for their definition, see \cite[Appendix]{db} or \cite[\S 2.3.4]{brs}). More precisely, from the discussion in \cite[p.~252]{bcy} we have that 
\begin{equation} \label{eqmajonz}
n_0\le 2(q+1)m^2(\mathcal{E}+(q+1)m+1),
\end{equation} 
where $\mathcal{E}$ is the maximum of all the modulus of the generalized local exponents at the infinite point and at the finite singularities of $A$.

At last we denote by $\kappa$ any real number such that 
$$0<\kappa\leq \max_{1\leq j \leq m}|f_j(\xi)|.$$

\begin{prop} \label{propmino}
There exists an effective constant $c$, which depends on $m, p, q, \xi$, $T(z)$, $\kappa$ and polynomially on $C$, $D$ and $E$, such that if $H\geq \max(3,n_0^{n_0})$ then
$$
\bigg\vert \sum_{j=1}^m a_j f_j(\xi) \bigg\vert > \frac{1}{H^c}.
$$
\end{prop}
Using Eq. \eqref{eqmajonz}, the lower bound on $H$ can be replaced by an explicit lower bound in terms of $m,q,\mathcal{E}$.

The important point in Proposition~\ref{propmino} is that $c$ depends on $f_1$, \ldots, $f_m$ only through a given set of parameters, and also that the dependence on $C$, $D$, $E$ is polynomial.

 The constant $c$ is effective because the only potential source of ineffectivity of the proof, i.e. $n_0$, is now known to be effective. In principle it would be possible to make $c$ completely explicit, but this would make the statement of our results much more complicated for no immediate application better than those we present here. Moreover such explicit formulas are not sharp in general.

In this proposition, and throughout this section, the polynomial dependence of $c$ with respect to $C,D,E$ means that there exists a polynomial $P\in\R[X,Y,Z]$ with non-negative coefficients and degree depending only on $m, p, q, \xi$, $T(z)$, $\kappa$, such that one may choose $c= P(C,D,E)$.
 
\bigskip

To prove Theorem~\ref{theo:2} we shall be interested in the following situation:
\begin{equation}\label{eqpartic}
 f_m(z)=\exp(\beta z) \mbox{ and $f_1,\ldots,f_{m-1}$ are independent from $\beta$.}
\end{equation}
In precise terms, when we refer to \eqref{eqpartic} we shall assume that $Z=\tra (f_1,\ldots,f_{m-1})$ is a vector of $E$-functions with rational coefficients, solution of a differential system $Z'=BZ$ with $B\in M_{m-1}(\Q(z))$. Then $Y=\tra (f_1,\ldots,f_{m})$ is solution of $Y'=AY$ where $A\in M_{m}(\Q(z))$ is blockwise diagonal, with diagonal blocks $B$ and $\beta$. The important point is that 
$f_1,\ldots,f_{m-1}$ and $B$ are independent from $\beta$. In this setting we have the following special case of Proposition~\ref{propmino}. For a given rational number $r\neq 0$, we set $\textup{den}(r)$ the positive denominator of $r$ written in reduced form. 

\begin{coro} \label{coromino}
In the situation \eqref{eqpartic},
there exists an effective constant $c$, which depends only on $f_1,\ldots,f_{m-1},\xi$ and polynomially on $|\beta|$ and $\textup{den}(\beta)$, such that if $H\geq \max(3,n_0^{n_0})$ then
$$
\bigg\vert \sum_{j=1}^m a_j f_j(\xi) \bigg\vert > \frac1{H^c}.
$$
Moreover the assumption on $H$  may be stated as a   lower bound    in terms of $f_1,\ldots,f_{m-1}$ only, independently of $\beta$.
\end{coro}

\begin{proof}
In the setting of \eqref{eqpartic}, recall that $Z=\tra (f_1,\ldots,f_{m-1})$ is a solution of $Z'=BZ$ with $B\in M_{m-1}(\Q(z))$, and $Y=\tra (f_1,\ldots,f_{m})$ of $Y'=AY$. The matrix $A\in M_{m}(\Q(z))$ is blockwise diagonal, with diagonal blocks $B$ and $\beta$. Therefore $T(z)$ is independent of $\beta$; so are $m, p, q$, and also 
$\mathcal{E}$ because the function $\exp(\beta z)$ has null generalized exponents everywhere. Hence, Eq. \eqref{eqmajonz} shows that $n_0$ can be bounded independently of $\beta$. 

Moreover at least one of $f_1$, \ldots, $f_{m-1}$ does not vanish at $\xi$, because $\xi$ is not a singularity of the differential system $Z'=BZ$ (indeed it is not a pole of a coefficient of $B$, because all these coefficients are coefficients of $A$); of course the functions $f_1$, \ldots, $f_{m-1}$ are not identically zero because they are linearly independent over $\Q(z)$. Therefore we may choose $\kappa = \max_{1\leq j \leq m-1}|f_j(\xi)|$.

Since $A_{m,m}=\beta$, we may take $C=\max(\widetilde C, \vert \beta\vert)$, $D=\textup{den}(\beta)\widetilde{D}$ and $E\leq |\beta|\widetilde{E}$, where $\widetilde{C}$, $\widetilde{D}$ and $\widetilde{E}$ are quantities analogous to $C$, $D$ and $E$ when we consider only $f_1, \ldots, f_{m-1}$. Therefore $C$, $D$, $E$ depend linearly on $|\beta|$ and $\textup{den}(\beta)$. Applying Proposition~\ref{propmino} concludes the proof of Corollary~\ref{coromino}.
\end{proof}

\begin{Remark} 
 In \cite[p.~421, Theorem~2]{shid}, Shidlovskii also proved for $E$-functions in the strict sense an effective~(\footnote{Again, strictly speaking, it was ineffective when Shidlovskii proved it because of the presence of $n_0$, but it is now effective.}) refinement of \eqref{eq:mesindepshid} in which he replaced the exponent $m-1+\varepsilon$ by 
 $m-1+\gamma m ^{7/2}(\ln\ln(H))^{-1/2}$ where $\ln\ln(H)\ge \gamma^2\max(m^2, \ln(n_0))$ and $\gamma$ is a constant. However the dependence of $\gamma$ with respect to $f_1$, \ldots, $f_m$ is unknown while to prove Theorem~\ref{theo:2}, it is necessary to know how the constants depend on $\beta$ in the setting of \eqref{eqpartic}.
\end{Remark}

\subsection{Shidlovskii's measure \eqref{eq:mesindepshid}: sketch of proof} \label{ssec:21}

We first recall the statement of \cite[p.~107, Lemma~14]{shid} when the number field $\mathbb K=\mathbb Q$. Assume $Y:={}^t(f_1, \ldots, f_m)$ is a vector of $E$-functions in Siegel's sense~(\footnote{However, we shall make this proof explicit only for $E$-functions in the strict sense, otherwise it would be difficult to obtain a good control on the quantities we are interested in, as for instance the parameters $C$ and $D$ do not exist for $E$-functions in Siegel's sense.}) in $\mathbb Q[[z]]$ solution of a differential system $Y'=AY$ with $A\in M_m(\mathbb Q(z))$. Let $n\in \mathbb N$ and $\varepsilon_1 \in (0, \frac1{2m-1})$; the reason of this technical assumption on $\varepsilon_1$ will appear in \S\ref{ssec:effecfinal}. There exist $P_1, \ldots, P_m\in \mathbb Z[z]$ of degree at most $n$ and not all zero such that: 
$$
 b_{i,\nu} 
 =\mathcal{O}(n^{(1+\varepsilon_1)n}), \quad 
i=1,\ldots, m, \; \nu =0,\ldots, n,
$$
where 
$b_{i,\nu}$ is the coefficient of $z^\nu$ in $P_i(z)$ and
the symbol $\mathcal{O}$ is uniform in $i$ and $\nu$, and such that the function
$$
R:=\sum_{i=1}^m P_i(z)f_i(z) =\sum_{\nu=\tau}^\infty \frac{a_\nu}{\nu!} z^\nu
$$
satisfies $\ord_{z=0} R(z)\ge \tau$ with 
$$\tau=m(n+1)-\lfloor \varepsilon_1 n\rfloor-1,$$
and $ a_\nu=\nu^{\varepsilon_1 n}\mathcal{O}(n^n)$ for $\nu\ge \tau$. 

\medskip

In the same setting, define $$R_k:=\sum_{i=1}^m P_{k,i}(z)f_i(z)$$ with $P_{k,i}\in \mathbb Z[z]$ by $R_1=R$ and $R_{k+1}=T(z)R_k'(z)$; recall that $ T\in \mathbb Z[z]\setminus\{0\}$ is a common denominator of minimal degree of the entries of $A$. 

From now on, we assume that $f_1, \ldots, f_m$ are $\mathbb Q(z)$-linearly independent, and that $n\geq n_0$ where $n_0$ was introduced in \S \ref{subsecenoncemesure}.

Then by Lemma~10 of \cite[p.~101]{shid}, for any $\xi\in \mathbb C$ such $\xi T(\xi)\neq0$, the linear forms $R_k(\xi)$, $k=1,\ldots, m+t_1$, include $m$ linearly independent forms, where
\begin{equation}\label{eq:defintiont}
t_1:=q\frac{(m-1)m}{2}+ \lfloor \varepsilon_1 n\rfloor+p;
\end{equation}
we recall from \S \ref{subsecenoncemesure} (and \cite[p.~93]{shid}) that
$$p:=\min_{1\le i \le m}\ord_{z=0}f_i(z)
\quad \mbox{ and } \quad 
q:
=\max(\deg T, \max_{i,j} \deg(TA_{i,j}))
$$
with $A:=(A_{i,j})_{1\leq i,j \leq m}$.
Now Lemma~15 in \cite[p.~110]{shid} says the following. Suppose that $\varepsilon_1=\varepsilon/(6(m+1))$ for some $\varepsilon\in (0,1)$. Then for any $\xi\in \Q$ such that $\xi T(\xi)\neq 0$, we have 
$$
R_k(\xi) = \mathcal{O}(n^{-(m-1-\varepsilon/2)n}), \quad k\le m+t_1,
$$
and
$$
P_{k,i}(\xi)=\mathcal{O}(n^{(1+\varepsilon/2)n}), \quad k\le m+t_1, \; i=1, \ldots, m.
$$
From these estimates, Shidlovskii deduces \cite[p.~357,~Theorem~1]{shid}, ie \eqref{eq:mesindepshid} for $E$-functions in Siegel's sense.

With Shidlovskii's original method, the constant $c$ in \eqref{eq:mesindepshid} was ineffective because of the ineffectivity of the integer $n_0$. As we have explained before, this is no longer the case. 
For our purpose, we need a different version of the measure \eqref{eq:mesindepshid}, with a  control of the dependencies of the constants on the parameters.

\subsection{Proof of Proposition \ref{propmino}} \label{ssec:22}

In this section we shall prove Proposition~\ref{propmino}.

In all what follows, as in \S \ref{subsecenoncemesure}, when we say that {\em a constant $c$ depends polynomially on a parameter $k$} we mean that there exists a polynomial $p$ with positive coefficients such that $\vert c\vert \le p(\vert k\vert)$. The polynomial $p$, including its degree, may depend on all other parameters $c$ depends on.

We shall follow now the sketch of proof given in the previous section; at each step we shall make all bounds explicit (to be precise, we shall make the dependencies in terms of the parameters explicit).

As in \S \ref{subsecenoncemesure} we consider a vector
 $Y={}^t(f_1, \ldots, f_m)$ of strict $E$-functions in $\mathbb Q[[z]]$, solution of a differential system $Y'=AY$ with $A\in M_m(\mathbb Q(z))$. We assume that $f_1, \ldots, f_m$ are $\mathbb Q(z)$-linearly independent, and denote by $ T\in \mathbb Z[z]\setminus\{0\}$ a common denominator of minimal degree of the entries of $A$. We fix $\xi\in \mathbb Q$ be such that $\xi T(\xi)\neq 0$. 
 As in \S \ref{ssec:21} we consider also $\varepsilon_1>0$ such that $\varepsilon_1< \frac1{2m-1}$.

\subsubsection{Construction of the polynomials} \label{subsubsec421}

Following the proof of \cite[p.~107, Lemma~14]{shid} and using Siegel's lemma, we find that for any $1\leq i \leq m $ and any $ 0\leq \nu\leq n$, 
$$
\vert b_{i,\nu}\vert \le n!2\big(m(n+1) CD\big)^{2m/\varepsilon_1}(2CD)^{4m^2 n/\varepsilon_1}
$$
and for any $\nu\geq\tau$,
 $$ 
\vert a_\nu \vert \le n!m 2^{\nu+1} 
\big(m(n+1) CD\big)^{2m/\varepsilon_1}(2CD)^{4m^2 n/\varepsilon_1} C^{\nu+1}.
 $$
Since 
$$
R(z)=\sum_{\nu=\tau}^\infty \frac{a_\nu}{\nu!}z^\nu,
$$
we deduce that 
\begin{equation*}
\vert R(z)\vert \le n!m \big(m(n+1)CD\big)^{2m/\varepsilon_1}(2CD)^{4m^2 n/\varepsilon_1} 2C \sum_{\nu=\tau}^\infty \frac{(2C\vert z\vert)^\nu}{\nu!}.
\end{equation*}
As 
$$
\sum_{\nu=\tau}^\infty \frac{t^\nu}{\nu!}=\frac{1}{(\tau-1)!}\int_0^t (t-x)^{\tau-1}e^xdx\leq \frac{ t^{\tau-1}e^t}{(\tau-1)!} \mbox{ for any }t\geq 0,
$$
it follows that 
$$
\vert R(z)\vert \le C_1(2CD)^{6m^2 n/\varepsilon_1} e^{2C\vert z\vert}
\frac{n!}{\big(m(n+1)-\lfloor \varepsilon_1 n\rfloor\big)!}(2C\vert z\vert)^{m(n+1)-\lfloor \varepsilon_1 n\rfloor-2} 
$$
where $C_1>0$ depends on $\varepsilon_1, m$, polynomially on $n$, not on $z$ and not on the Taylor coefficients of the $f_j$'s. This constant $C_1$ also satisfies 
\begin{equation} \label{eqmajoanu}
\vert a_\nu \vert \le n! \, C_1 \, 
(2CD)^{6m^2 n/\varepsilon_1} \, (2C)^{\nu+1} \mbox{ for any } \nu\geq\tau
\end{equation}
 and 
\begin{equation} \label{eqmajobknu}
\vert b_{i,\nu}\vert \le n!\, C_1 \, 
(2CD)^{6m^2 n/\varepsilon_1} \mbox{ for any }1\leq i \leq m 
\mbox{ and any } 0\leq \nu\leq n.
\end{equation}

\bigskip

\subsubsection{Upper bounds on the linear forms} \label{subsubsec423}

By a similar analysis of the proof of \cite[p.~110, Lemma~15]{shid}, using Eq. \eqref{eqmajoanu} we see that 
$$
\vert R_k(z) \vert \le C_2 (2CD)^{6m^2n/\varepsilon_1} 
\frac{n!k!(2q)^ke^{2C\vert z\vert} (1+|z|)^{(k-1)q}}{\big(m(n+1)-\lfloor \varepsilon_1 n\rfloor-k-2\big)!} 
(2C\vert z\vert)^{m(n+1)-\lfloor \varepsilon_1 n\rfloor-k-2}
$$
for all $n\ge n_0$ and all $k\in \{1, \ldots, m+t_1\}$, where $C_2>0$ 
depends on $\varepsilon_1, m,p,q$, polynomially on $n$, linearly on 
the $k$-th power of the maximum modulus of the coefficients of the polynomial $T(z)$, and not on $z$.

Moreover, the degree of each $P_{k,i}$ is less than $n+(k-1)q$ by \cite[p.~115]{shid} and using Eq.~\eqref{eqmajobknu} we have 
$$
\vert P_{k,i}(z)\vert \le C_3 (2CD)^{6m^2 n/\varepsilon_1}(1+\vert z\vert)^{n+(k-1)q}(m+n)^k n!k!q^{k} , \quad k\le m+t_1, \; i=1, \ldots, m
$$
where $C_3$ depends on $\varepsilon, m,p,q$, polynomially on $n$ and $z$, and linearly on $E^k$. As a polynomial in $n$ and $z$, the degree of $C_3$ depends only on $\varepsilon_1, m,p,q$.

\medskip

 In the above upper bounds for $R_k$ and $P_{k,i}$, we now use the fact that $k\le m+t_1\le \lfloor \varepsilon_1 n\rfloor+C_4$, where $C_4$ depends only on $m$, $p$, $q$ (by Eq.~\eqref{eq:defintiont}), but not on $n$ or $z$. As in \cite{shid} we take $z=\xi$ and multiply by a common denominator.
 After some simplifications, 
the situation can now be summarized as follows (this makes explicit \cite[p.~114,~Lem\-ma~16]{shid}): for every $\xi\in \mathbb Q$ such that $\xi T(\xi)\neq 0$ and for all $n\ge n_0$, there exist $m$ linearly independent linear forms 
$$
L_j:=\sum_{i=1}^m a_{j,i} f_i(\xi), \quad a_{j,i}\in \mathbb Z, \quad j=1, \ldots, m,
$$
(that depend on $n$) such that for any $1\leq j \leq m$,
\begin{multline} \label{eq:majLk}
\quad \vert L_j\vert \le C_5 (2CD)^{6m^2 n/\varepsilon_1}{e^{2C\vert \xi\vert}}(2q\max(\varepsilon_1,C_4))^{\varepsilon_1 n}\\
\times 
 (1+|\xi|)^{q\varepsilon_1n} \textup{den}(\xi) ^{(1+q\varepsilon_1)n} \,\frac{n^{n+\varepsilon_1 n}(2C\vert \xi\vert)^{mn-2\varepsilon_1 n}}
{(mn-2\lfloor \varepsilon_1 n\rfloor)!}, 
\end{multline}
where the factor $C_5 $ depends on $\varepsilon_1, m,p,q, \xi$, polynomially in $n$ and linearly on the $(m+t_1)$-th power of the maximum modulus of the coefficients of the polynomial $T(z)$. Since this exponent is $\le \varepsilon_1 n+C_4$ we have 
$C_5\le \widetilde{c_5}^n$ where $\widetilde{c_5}$ depends on $\varepsilon_1, m,p,q, \xi, T(z)$ but not on $n$. 
Moreover, for any $1\leq i,j\leq m$, 
\begin{equation}\label{eq:majaki}
\quad \vert a_{j,i}\vert \le C_6 (2CD)^{6m^2 n/\varepsilon_1}(2q\max(\varepsilon_1, C_4))^{\varepsilon_1 n} 
\big((1+|\xi|)\textup{den}(\xi)\big)^{(1+q\varepsilon_1)n}
n^{n+2\varepsilon_1 n}, 
\end{equation}
where $C_6 $ depends on $\varepsilon_1, m, p, q, \xi$, 
polynomially in $n$ and linearly in $E^{m+t_1}$. 
As a polynomial in $n$, the degree of $C_6$ depends only on $\varepsilon_1, m, p, q$.

\subsubsection{Conclusion} \label{ssec:effecfinal}

We are now ready to analyze the proof of \cite[p.~357,~Theorem~1]{shid} in the case $\mathbb K=\mathbb Q$ in order to make Shidlovskii's measure explicit. Shidlovskii proves that for any $\varepsilon\in (0,1)$, any integer $H\ge 1$ and any vector $(a_1, \ldots, a_m)\in \mathbb Z^m\setminus \{0\}$ such that $\max\vert a_i\vert \le H$, we have
$$
\bigg\vert\sum_{i=1}^m a_i f_i(\xi) \bigg\vert \ge b_4 n^{-(m-1+\varepsilon)n}\big(1-b_5 Hn^{-(1-\varepsilon)n}\big),
$$
where $b_4,b_5>0$ are not computed but are known to be independent of $H$. He then chooses the smallest $n$ such that $n\ge n_0$ and $n^{(1-\varepsilon)n}>2b_5H$ (such an $n$ obviously exists) to deduce the expected linear independence measure: 
$$
\bigg\vert\sum_{i=1}^m a_i f_i(\xi) \bigg\vert \ge \frac{b_7}{H^{m-1+2m\varepsilon}}
$$
for a constant $b_7>0$ which is again not computed but is known to be independent of $H$. 

\medskip

In the proof of \cite[p.~357,~Theorem~1]{shid} in the case $\mathbb K=\mathbb Q$, we have $i=h=1$. Shidlovskii defines the determinant $\Delta$ whose entries are the coefficients (in $\mathbb Z$) of the $m$ linear forms $L_1, \ldots, L_{m-1}, L_0:=\sum_{i=1}^m a_if_i(\xi)$ (which up to reordering can be assumed to be linearly independent without loss of generality when $n\ge n_0$), and the determinant $\Delta_{j,\ell}$ which is the cofactor of the entry in the $j$-th row and $\ell$-column of $\Delta$. (Each line corresponds to a linear form, with $L_1$ at the top and $L_0$ at the bottom of the determinant.) 
In \cite[p.~358, Eq.~(41)]{shid}, each occurence of $i(=1)$ can be deleted and we have
\begin{equation}\label{eq:lowerbound1}
 |\Delta_{m,\ell}|\cdot \vert L_0 \vert
\ge \vert f_{\ell}(\xi)\vert \cdot \vert \Delta\vert-(m-1)\max_{1\le j\le m-1} |\Delta_{j,\ell} |\cdot \max_{1\le j\le m-1} \vert L_j\vert.
\end{equation}
This inequality holds for any $\ell \in\{1, \ldots, m\}$ such that $f_\ell(\xi)\neq 0$; such an $\ell$ exists because $\xi T(\xi)\neq 0$. We have $\vert \Delta\vert \ge 1$ when $n\ge n_0$. 
Then using \eqref{eq:majLk} and \eqref{eq:majaki}, for any $n\ge n_0$, we can replace the three bounds in \cite[p.~359, Eq.~(42)]{shid} by 
\begin{multline*}
\max_{1\le j\le m-1} \vert L_{j}\vert \le 
C_5 (2CD)^{6m^2 n/\varepsilon_1}{e^{2C\vert \xi\vert}}(2q\max(\varepsilon_1,C_4))^{\varepsilon_1 n}
\\
\times 
(1+|\xi|)^{q\varepsilon_1 n}\textup{den}(\xi)^{(1+q\varepsilon_1)n} \,\frac{n^{n+\varepsilon_1 n}(2C\vert \xi\vert)^{mn-2\varepsilon_1 n}}
{(mn-2\lfloor \varepsilon_1 n\rfloor)!},
\end{multline*}
$$
\housesurQ{\Delta_{m,\ell}} \le C_7 \Big((2CD)^{6m^2 n/\varepsilon_1}(2q\max(\varepsilon_1, C_4))^{\varepsilon_1 n} 
\big((1+|\xi|)\textup{den}(\xi)\big)^{(1+q\varepsilon_1)n}
n^{n+2\varepsilon_1 n}
\Big)^{m-1}
$$
and 
$$
\max_{1\le j\le m-1}\housesurQ{\Delta_{j,\ell}} \le C_8 H \Big(
(2CD)^{6m^2 n/\varepsilon_1}(2q\max(\varepsilon_1, C_4))^{\varepsilon_1 n} 
\big((1+|\xi|)\textup{den}(\xi)\big)^{(1+q\varepsilon_1)n}
n^{n+2\varepsilon_1 n}
\Big)^{m-2},
$$
where $C_7, C_8>0$ have the same dependencies as $C_6^m$, where $C_6$ is considered in \eqref{eq:majaki}, and both can be bounded accordingly. We recall from  \S \ref{subsubsec423} that:
\begin{itemize}
\item[$\bullet$] $C_4$ depends only on $m$, $p$, $q$, but not on $n$, $\xi$.
\item[$\bullet$] $C_5\le \widetilde{c_5}^{n}$ where $\widetilde{c}_5$ depends on $\varepsilon_1, m,p,q, \xi, T(z)$ but not on $n$.
\item[$\bullet$] $C_6\le \widetilde{c_6}^{n} $ where $\widetilde{c_6}$ depends on $\varepsilon_1, m, p, q, \xi$, and polynomially on $E$, but is independent of $n$. All polynomials involved here have degrees bounded in terms of $\varepsilon_1, m, p, q$.
\end{itemize}

Now the three bounds obtained above yield, for any $n\ge n_0$,
$$
\max_{1\le j\le m-1} \vert L_{j}\vert \le C_9 {e^{2C\vert \xi\vert}} n^{C_{10}} C_{11}^n n^{-n(m-1- 3\varepsilon_1)}
$$
$$
\housesurQ{\Delta_{m,\ell}} \le C_9 n^{C_{10}} C_{11}^n 
n^{ n(m-1)(1+2\varepsilon_1)}, 
$$
and
$$
\max_{1\le j\le m-1}\housesurQ{\Delta_{j,\ell}} \le C_9 H n^{C_{10}} C_{11}^n n^{ n(m-2)(1+2\varepsilon_1)},
$$
where :
\begin{enumerate}
\item[$\bullet$] $C_9\ge 1$ depends on $\varepsilon_1, m, p, q, \xi$, polynomially on $C$ and $D$ but not on $n$. 
\item[$\bullet$] $C_{10}\ge 0$ depends on $\varepsilon_1, m, p, q, \xi $, not on $n$ and not on the Taylor coefficients of the $f_j$'s. 
\item[$\bullet$] $C_{11}\ge 1$ depends on $\varepsilon_1, m, p, q, \xi$, $T(z)$, polynomially on $C$, $D$ and $E$ but not on $n$. 
\item[$\bullet$] The degree of $C_9$ as a polynomial in $C$ and $D$, and the one of $C_{11}$ in $C$, $D$ and $E$, are bounded in terms of $\varepsilon_1, m, p, q$. 
\end{enumerate} 

Now we choose $\ell$ such that 
$$ \vert f_\ell(\xi)\vert = \max_{1\leq j \leq m}|f_j(\xi)| \geq \kappa$$
and we deduce from \eqref{eq:lowerbound1} that, for $n\ge n_0$,
\begin{equation}\label{eq:lowerbound2}
\vert L_0 \vert \ge \frac{\kappa}{C_9 n^{C_{10}}C_{11}^n 
n^{ n(m-1)(1+2\varepsilon_1)}
}
\left(1-un^v w^n n^{-\delta n} H
\right)
\end{equation}
where 
$$u:=\max\big(1,(m-1)C_9^2{e^{2C\vert \xi\vert}}/ \kappa \big),
\quad v:=2C_{10}, \quad w:=C_{11}^2,$$ 
and $ \delta:=1-(2m-1)\varepsilon_1 \in (0,1)$
by the assumption made on $\varepsilon_1$ at the beginning of \S\ref{ssec:22}. 
This implies
$$
\vert L_0 \vert \ge \frac{\kappa}{2C_9 n^{C_{10}}C_{11}^n 
n^{ n(m-1)(1+2\varepsilon_1)}}
$$
provided 
\begin{equation}\label{eq:lowerbound3}
n \ge n_0 \quad \textup{and} \quad 2uH\le w^{-n} n^{\delta n-v}.
\end{equation}

Since $u\geq 1$ and we assume
$H\geq \max(3,n_0^{n_0})$ in Proposition~\ref{propmino}, we have 
 $2uH \geq n_0^{n_0}$. Accordingly for any $n<n_0$ we have $2uH>n^n \geq w^{-n}n^{\delta n-v}$, so that the minimal value of $n$ (denoted by $N$ from now on) that satisfies $2uH\le w^{-n} n^{\delta n-v}$ is automatically $\ge n_0$: it satisfies the assumptions \eqref{eq:lowerbound3}.
 
We want to find an upper bound for $N$ in terms of $2uH,w,v,\delta$. An equivalent definition of $N$ is that it is the largest integer such that
\begin{equation} \label{eq:definitionN}
 (N-1)^{N-1}<(2u H)^{1/\delta}(w^{1/\delta})^{N-1}(N-1)^{v/\delta}.
\end{equation}
For any $X\ge 1$ we have the following lower bounds:
\begin{align*}
\frac{X^X}{(2uH)^{1/\delta} w^{X/\delta} X^{v/\delta}}
 &\geq 
 \frac{X^X}{(2uH)^{1/\delta} X^{X/4}X^{v/\delta}} \quad \textup{assuming that} \; X\ge w^{4/\delta}, 
 \\
 &= \frac{X^{3X/4}}{(2uH)^{1/\delta} X^{v/\delta}}
 \\
 &\ge \frac{X^{X/4}}{(2uH)^{1/\delta}} \quad \textup{assuming that} \; X\ge \frac{2v}{\delta}
 \\
 &=\left(\frac{X^X}{(2uH)^{4/\delta}}\right)^{1/4}
 \\
 &\ge 1 \quad \textup{assuming that} \; X\ge \frac{8\ln(2uH)}{\delta \ln\ln((2uH)^ {4/\delta})},
\end{align*}
where in the last line we use 
the elementary fact that for any $y>e$, if $x\ge \frac{2\ln(y)}{\ln\ln(y)}$, then $x^x\ge y$. Note that $H\geq 3$ and $u\geq 1$ so that $(2uH)^{4/\delta}\ge 6^{4/\delta}>e$ because $\delta\in (0,1)$ and thus we can use this fact with $y:=(2uH)^{4/\delta}$. All   assumptions in the lower bound above are satisfied for instance for any 
$$
X\geq 1+w^{4/\delta}+\frac{2v}{\delta}+\frac{8\ln(2uH)}{\delta\ln\ln((2uH)^{4/\delta})}. 
$$
Since for $X=N-1$ the lower bound $\frac{X^X}{(2uH)^{1/\delta} w^{X/\delta} X^{v/\delta}}\geq 1$ does not hold, we deduce that
\begin{equation}\label{eq:majorationN}
N\le 1+w^{4/\delta}+\frac{2v}{\delta}+\frac{8\ln(2uH)}{\delta\ln\ln((2uH)^ {4/\delta})}=:\Phi.
\end{equation}
Hence, since \eqref{eq:lowerbound3} holds  with $n:=N$, we have
\begin{align*}
\vert L_0 \vert &\ge \frac{\kappa}{2C_9 N^{C_{10}}C_{11}^N 
N^{ N(m-1)(1+2\varepsilon_1)}} 
\\ 
&=\frac{\kappa}{2C_9 e^{C_{10}\ln(N)+ \ln(C_{11})N +
 (m-1)(1+2\varepsilon_1)N\ln(N)}}
\\
&\geq \frac{1}{H^c} 
\end{align*}
because 
$$
N\ln(N)\leq \Phi \ln(\Phi)\leq C_{12} \ln(H) \mbox{ and } H\geq 3,
$$
where $c $ and $C_{12} $ depend on $\varepsilon_1, m, p, q, \xi, \kappa$, $T(z)$, and polynomially on $C$, $D$ and $E$. The degree of these polynomials are bounded in terms of $\varepsilon_1, m, p, q$. To conclude the proof of Proposition~\ref{propmino}, we choose $\varepsilon_1=\frac1{2m}$.

\section{Proof of Theorem~\ref{theo:2}} \label{sec:3}

\subsection{First reductions} \label{subsecreductions}

Let $f\in \mathbb Q[[z]]$ be a strict $E$-function, and 
 $\xi\in \mathbb Q^*$ be such that $f(\xi)>0$. Considering $f(\xi z) $ instead of $f$, we may assume that $\xi=1$. Recall  that $f(1)\neq 1$,  because $\ln(f(1))\notin \mathbb Q$ is an assumption of Theorem \ref{theo:2}. Accordingly, if $f(1)$ is algebraic then $\ln(f(1))$ is not a Liouville number by \cite[p.~386, Theorem~3]{mahler2}, and the conclusion of Theorem~\ref{theo:2} follows at once. Therefore we may assume that $f(1)$ is transcendental and apply the following consequence of \cite[Proposition~2]{muegaledeux}, which is a variant of Beukers' desingularization
 lemma \cite[Theorem~1.5]{beukers}.
 
\begin{prop} \label{propdesing}
Let $g_1$, \ldots, $g_k$ be $E$-functions with rational coefficients, such that $1$, $g_1$, \ldots, $g_k$ are linearly independent over $\Q(z)$ and $g_1(1)$ is transcendental. Assume also that the vector $\tra (1, g_1, \ldots, g_k)$ is solution of a differential system $Y'=AY$ with $A\in M_{k+1}(\Q(z))$.

Then there exist $E$-functions $f_1, \ldots, f_k$ with rational coefficients, such that $1, f_1, \ldots, f_k$ are linearly independent over $\Q(z)$, $f_1(1)=g_1(1)$, and the vector $\tra (1, f_1, \ldots, f_k)$ is solution of a differential system $Y'=BY$ with $B\in M_{k+1}(\Q[z, \frac1z])$. 

Moreover, when the $g_j$'s are strict $E$-functions, then the $f_j$'s are strict $E$-functions as well.
\end{prop}

To apply Proposition~\ref{propdesing}, we consider the inhomogeneous differential equation of minimal order satisfied by $f$. We denote its order by $\mu-1$, with $\mu\geq 2$ because $f$ is transcendental (since $f(1)$ is).
The functions $1, f, f', \ldots, f^{(\mu-2)}$ are linearly independent over $\Q(z)$: Proposition~\ref{propdesing} provides $E$-functions $f_1$, \ldots, $f_{\mu-1}$ with rational coefficients such that $f_1(1)=f(1)$ is the number we are interested in the logarithm of. Letting $f_\mu=1$, the functions $f_1$, \ldots, $f_\mu$ are linearly independent over $\Q(z)$ and make up a vector solution of a differential system without non-zero singularity; in particular 1 is not a singularity.

Now recall that the functions $\exp(\beta z)$, $\beta\in\Qbar$, are linearly independent over $\Q(z)$. Therefore at most $\mu$ of them belong to the $\Q(z)$-vector space spanned by $f_1$, \ldots, $f_\mu$. Since $\ln(f(1))$ is irrational, we may exclude these finitely many values of $\beta=a/b $ in proving the lower bound \eqref{eq:irratmeasurelogE} (up to changing the values of $c$ and $d$). Therefore we may restrict to rationals $\beta$ such that $f_1$, \ldots, $f_\mu$, and $\exp(\beta z)$ are linearly independent over $\Q(z)$.

\subsection{Application of the effective linear independence measure} \label{ssec:23}

As explained in the previous section, we are trying to bound 
$$ \left\vert \ln(f_1(1))-\frac{a}{b}\right\vert$$
from below, with $a\in\Z$ and $b\in \N^*$.
By the mean value theorem (see Eq. \eqref{eq:minlogf} below), it is essentially equivalent to bounding below 
$$ \left\vert f_1(1)-\exp( a/b )\right\vert,$$
which is a $\Z$-linear combination of the values at $1$ of the $E$-functions $f_1(z)$ and $\exp(\beta z)$, with $\beta := a/b$. We point out that the coefficients of this linear combination are only $0$, $1$ and $-1$ whereas in general, we are always interested in linear combinations with large coefficients.

As explained at the end of 
\S \ref{subsecreductions}, we may assume that 
$f_1, \ldots, f_\mu$ and $f_{\mu+1}(z)=\exp(\beta z)$ are strict $E$-functions linearly independent over $\Q(z)$, solution of a first-order differential system without non-zero singularity. Moreover $f_1, \ldots, f_\mu$ are solution of a first-order differential system without non-zero singularity, independent from $\beta$: we are in the setting of \eqref{eqpartic}. 

Corollary~\ref{coromino} yields a constant $c$, which depends on $f_1, \ldots, f_{m-1}$ and polynomially on $|\beta|$ and $\textup{den}(\beta)$, such that with $H= \max(3,n_0^{n_0})$ we have
$$
 \left\vert f_1(1)-\exp( a/b )\right\vert > \frac1{H^c}.
$$
An important feature of this corollary is that this value of $H$ depends only on $f_1, \ldots, f_{m-1}$, not on $\beta$.
 
If $\vert f_1(1)-\exp(a/b)\vert \geq f_1(1)/2 $, the lower bound of Theorem~\ref{theo:2} holds trivially. Therefore we may assume that 
\begin{equation}\label{eq:mineta}
\big\vert f_1(1)-\exp(a/b) \big\vert < 
f_1(1)/2.
\end{equation}
Accordingly $|\beta|$ is bounded in terms of $f_1(1)$ only, and there exists a polynomial $Q\in\R[X]$ with positive coefficients such that $c\leq Q(b)$; this polynomial $Q$ depends only on $f_1,\ldots,f_{m-1}$. Therefore
$$
 \left\vert f_1(1)-\exp( a/b )\right\vert > \frac1{H^{Q(b)}}
$$
and this concludes the proof of Theorem~\ref{theo:2}. Indeed by the mean value theorem, for all $(a,b)\in \mathbb Z^*\times \mathbb N^*$, $a,b$ coprime, there exists $\omega>0$ in the interval with endpoints $\exp(a/b)$ and $f_1(1)$ such that 
\begin{equation}
\label{eq:minlogf}
\left\vert \ln(f_1(1))-\frac{a}{b} \right\vert =\frac{1}{\omega} \left\vert f_1(1)-\exp(a/b)\right\vert, 
\end{equation}
and finally in this equality, the coefficient $\omega$ can be bounded above in terms of $f_1(1)$ only due to Eq. \eqref{eq:mineta}.

\noindent St\'ephane Fischler, Universit\'e Paris-Saclay, CNRS, Laboratoire de math\'ematiques d'Orsay, 91405 Orsay, France.

\medskip

\noindent Tanguy Rivoal, Universit\'e Grenoble Alpes, CNRS, Institut Fourier, CS 40700, 38058 Grenoble cedex 9, France.

\bigskip

\noindent Keywords: $E$-functions, Irrationality measure, Shidlovskii's linear indepence measure

\bigskip

\noindent MSC 2020: 11J82 (Primary), 11J91 (Secondary)

\end{document}